\documentclass{amsart}

\usepackage[latin1]{inputenc}
\usepackage[T1]{fontenc}
\usepackage{enumerate}
\usepackage{mathtools}
\usepackage{amssymb}
\usepackage{mathrsfs}
\usepackage{mathabx}

\newtheorem{theorem}{Theorem}[section]
\newtheorem{lemma}[theorem]{Lemma}
\newtheorem{cor}[theorem]{Corollary}
\newtheorem{prop}[theorem]{Proposition}

\theoremstyle{definition}
\newtheorem{definition}[theorem]{Definition}

\theoremstyle{remark}

\numberwithin{equation}{section}

\newcommand{\N}{\mathbb N}

\newcommand{\T}{\tilde{T}}

\title{Inverse of frequently hypercyclic operators}
\author[Q. Menet]{Quentin Menet}

\address{Quentin Menet, Département de Mathématique, Université de Mons, 20 Place du Parc, Mons, Belgique}
\email{quentin.menet@umons.ac.be}
\thanks{Quentin Menet is a Research Associate of the Fonds de la
Recherche Scientifique - FNRS and was supported by the grant ANR-17-CE40-0021 of the French National 
Research Agency ANR (project Front)}

\allowdisplaybreaks[3]
\begin{document}
\begin{abstract}
We show that there exists an invertible frequently hypercyclic operator on $\ell^1(\mathbb{N})$ whose inverse is not frequently hypercyclic.
\end{abstract}
\maketitle
\section{Introduction}

Given $X$ a separable infinite-dimensional Banach space and $T$ a continuous and linear operator on $X$, we can consider, for each vector $x\in X$, the set  $\text{Orb}(x,T)=\{n\ge 0:T^n x\}$ which is called the orbit of $x$ under the action of $T$. Linear dynamics is the theory studying the properties of such orbits. One of the basic notions in linear dynamics is the hypercyclicity. An operator $T$ is said to be hypercyclic if there exists a vector $x\in X$ such that $\text{Orb}(x,T)$ is dense in $X$, or equivalently, such that for every non-empty open set $U \subset X$, the set $N_{T}(x,U):=\{n\ge 0:T^n x\in U\}$ is non-empty (or equivalently infinite). Several important notions related to hypercyclicity have been introduced and deeply investigated during the last decades. We will mention some of them in this paper but for more information concerning linear dynamics, the reader can refer to two books~\cite{Book1, Book2}.

While there is no hypercyclic operator in finite-dimension, each separable infinite-dimensional Banach space supports a hypercyclic operator~\cite{Ansari, Bernal}. We can wonder if it is possible to require more on the sets $N_T(x,U)$. In 2004, Bayart and Grivaux~\cite{BG1, BG2} have introduced the notion of frequent hypercyclicity. An operator $T$ is said to be frequently hypercyclic if there exists a vector $x\in X$ such that for every non-empty open set $U\subset X$, $\underline{\text{dens}}(N_T(x,U))>0$. In these papers, Bayart and Grivaux gave sufficient conditions for frequent hypercyclicity and showed that there are simple frequently hypercyclic operators on each space $\ell^p(\mathbb{N})$ (with $1\le p<\infty$). However, there exist separable infinite-dimensional Banach spaces supporting no frequently hypercyclic operator and even supporting no $\mathcal{U}$-frequently hypercyclic operator~\cite{Shkarin}. An operator $T$ is said to be $\mathcal{U}$-frequently hypercyclic if there exists a vector $x\in X$ such that for every non-empty open set $U\subset X$, $\overline{\text{dens}}(N_T(x,U))>0$. 

Several open questions concerning frequent hypercyclicity posed in \cite{BG2, BG3} have been challenging for many years. One of them, recently solved, concerned the link between chaos and frequent hypercyclicity. We recall that an operator $T$ is said to be chaotic if $T$ is hypercyclic and possesses a dense set of periodic points. Indeed, Bayart and Grivaux~\cite{BG3} showed in 2007 that there exists a frequently hypercyclic weighted shift on $c_0$ which is not chaotic while a chaotic operator which is not frequently hypercyclic was obtained in 2017~\cite{Menet}. This last counterexample required the introduction of a new family of operators, called operators of $C$-type, which have been deeply investigated in \cite{Monster}.

The purpose of this paper is to answer the following open question which can be found in \cite{BG2, Ruzsa, Grosse, Guirao}:
\begin{center}
Does the inverse of an invertible frequently hypercyclic operator is frequently hypercyclic?
\end{center}

It is well-known that the inverse of an invertible hypercyclic operator is always hypercyclic. It is a direct consequence of Birkhoff transitivity theorem. For $\mathcal{U}$-frequent hypercyclicity, it was recently proved that it is not the case anymore; there exists an invertible $\mathcal{U}$-frequently hypercyclic operator on $\ell^p(\mathbb{N})$ (with $1\le p<\infty$) whose inverse is not $\mathcal{U}$-frequently hypercyclic~\cite{Menet2}. This counterexample was obtained by considering a suitable operator of $C$-type. However, we know that if $T$ is invertible and frequently hypercyclic then $T^{-1}$ is $\mathcal{U}$-frequently hypercyclic~\cite{Ruzsa}. It means that if we want to exhibit a frequently hypercyclic operator $T$ whose inverse is not frequently hypercyclic then we first need to find a $\mathcal{U}$-frequently hypercyclic operator which is not frequently hypercyclic. Such operators exist~\cite{Ruzsa, Monster} but each of known examples is clearly not invertible. We will show in this paper that there exists an invertible frequently hypercyclic operator on $\ell^1(\mathbb{N})$ whose inverse is not frequently hypercyclic by introducing a generalization of operators of $C$-type.

\section{Generalization of operators of $C$-type}

Operators of $C$-type are associated to four parameters $v$, $w$, $\varphi $ and $b$, where

\begin{enumerate}
 \item[-] $v=(v_{n})_{n\ge 1}$ is a bounded sequence of non-zero complex numbers;
\item[-] $w=(w_{j})_{j\geq 1}$ is a sequence of complex numbers which is both bounded  
and bounded below, \mbox{\it i.e.} $0<\inf_{k\ge 1} \vert w_k\vert\leq \sup_{k\ge 1}\vert w_k\vert<\infty$,
\item[-] $\varphi $ is a map from $\N$ into itself, such that $\varphi 
(0)=0$, $\varphi (n)<n$ for every $n\ge 1$, and the set 
$\varphi ^{-1}(l)=\{n\ge 0\,:\,\varphi (n)=l\}$ is infinite for every 
$l\ge 0$;
\item[-] $b=(b_{n})_{n\ge 0}$ is a strictly increasing sequence of positive 
integers such that $b_{0}=0$ and $b_{n+1}-b_{n}$ is a multiple of 
$2(b_{\varphi (n)+1}-b_{\varphi (n)})$ for every $n\ge 1$.
\end{enumerate}

These operators are then defined as follows:

\begin{definition}\label{Definition 43}
Let $W_n=\prod_{b_n<j<b_{n+1}}  w_j$ such that $\inf_{n\geq 0} |W_n| >0$.
 The \emph{operator of $C$-type} $T_{v,w,\varphi,b}$ on $\ell^1(\N)$
 associated to the data $v$, $w$, $\varphi $, and $b$ given as above is defined by
\[
T_{v,w,\varphi,b}\ e_k=
\begin{cases}
 w_{k+1}\, e_{k+1} & \textrm{if}\ k\in [b_{n},b_{n+1}-1),\; n\geq 0,\\
v_{n}\, e_{b_{\varphi(n)}}-W_n^{ -1 } e_{
b_{n}} & \textrm{if}\ k=b_{n+1}-1,\ n\ge 1,\\
 -W_0^{-1}e_0& \textrm{if}\ 
k=b_1-1,
\end{cases}
\]
where $(e_k)_{k\ge 0}$ is the canonical basis of $\ell^1(\mathbb{N})$.
\end{definition}

These operators have the nice property that every finite sequence is periodic. More precisely, for every $k\in[b_n,b_{n+1})$, we have
\[T_{v,w,\varphi,b}^{2(b_{n+1}-b_n)}e_k=e_k\quad  (\text{see \cite[Fact 6.4]{Monster}}).\]
Periodic points are in general quite helpful for studying dynamical properties of an operator. For instance, a simple criterion for frequent hypercyclicity based on the behaviour of periodic points was given in \cite[Theorem 5.31]{Monster}. Unfortunately, if periodic points of an invertible operator $T$ satisfy the conditions of this criterion then periodic points of $T^{-1}$ (which are periodic points of $T$) will also satisfy these conditions. Therefore, we cannot use this criterion in order to establish the frequent hypercyclicity of our counterexample and it seems better to perturb these periodic points to prevent frequent hypercyclicity being transmitted to its inverse.

For this reason, we will introduce a new family of operators which contains operators of $C$-type but also operators for which finite sequences are not periodic.

\begin{definition}\label{Definition 43}
Let $R=(R_n)_{n\ge 0}$ with $\inf_{n\geq 0} |R_n| >0$.
The operator $T_{v,w,\varphi,b,R}$ on $\ell^1(\N)$
 associated to the data $v$, $w$, $\varphi $, $b$ and $R$ given as previously is defined by
\[
T_{v,w,\varphi,b,R}\ e_k=
\begin{cases}
 w_{k+1}\, e_{k+1} & \textrm{if}\ k\in [b_{n},b_{n+1}-1),\; n\geq 0,\\
v_{n}\, e_{b_{\varphi(n)}}-R_n^{ -1 } e_{
b_{n}} & \textrm{if}\ k=b_{n+1}-1,\ n\ge 1,\\
 -R_0^{-1}e_0& \textrm{if}\ 
k=b_1-1.
\end{cases}
\]
\end{definition}

The operator $T_{v,w,\varphi,b,R}$ is thus an operator of $C$-type as soon as $R_n=W_n$ for every $n$. A direct consequence of this generalization lies in the fact that the elements in $c_{00}$ are in general not periodic for $T_{v,w,\varphi,b,R}$. However, since we want to deduce dynamical properties of $T_{v,w,\varphi,b,R}$ by investigating the behaviour of orbits of finite sequences, we would like that orbits of finite sequences remain simple to study. For this reason, we first show that under some additional conditions, every vector $e_k$ is an eigenvector of $T^n$ for some $n$. This is the purpose of the following lemma.

\begin{lemma}\label{Lem}
Let $W_n=\prod_{b_n<j<b_{n+1}}  w_j$ and $R=(R_n)_{n\ge 0}$ with $\inf_{n\geq 0} |R_n| >0$. If for every $n\ge 1$, we have 
\[R_n^{-1} W_n=(R^{-1}_{\varphi(n)}W_{\varphi(n)})^{\frac{b_{n+1}-b_n}{b_{\varphi(n)+1}-b_{\varphi(n)}}},\] then for every $n\ge 0$, every $k\in[b_n,b_{n+1})$,
\[T_{v,w,\varphi,b,R}^{2(b_{n+1}-b_n)}e_k=(R^{-1}_nW_n)^2 e_k.\]
\end{lemma}
\begin{proof}
We remark that it suffices to show that for every $n\ge 0$, 
\[T_{v,w,\varphi,b,R}^{2(b_{n+1}-b_n)}e_{b_n}=(R^{-1}_nW_n)^2 e_{b_n}\]
since for every $k\in[b_n,b_{n+1})$, $e_k$ is a multiple of $T^{k-b_n}e_{b_n}$.

It follows from the definition of $T_{v,w,\varphi,b,R}$ that \[T_{v,w,\varphi,b,R}^{2(b_{1}-b_0)}e_{b_0}=(R^{-1}_0W_0)^2e_{b_0}.\]
Let $N\ge 1$. Assume that we have $T_{v,w,\varphi,b,R}^{2(b_{n+1}-b_n)}e_{b_n}=(R_n^{-1}W_n)^2 e_{b_n}$ for every $n<N$. Then
\[T_{v,w,\varphi,b,R}^{b_{N+1}-b_N}e_{b_N}=-R_N^{-1}W_N e_{b_N}+v_N W_N e_{b_{\varphi(N)}}\]
and since $b_{N+1}-b_N$ is a multiple of $2(b_{\varphi(N)+1}-b_{\varphi(N)})$, we have by induction hypothesis
\begin{align*}
T_{v,w,\varphi,b,R}^{2(b_{N+1}-b_N)}e_{b_N}&=T_{v,w,\varphi,b,R}^{b_{N+1}-b_N}\left(-R_N^{-1}W_N e_{b_N}+v_N W_N e_{b_{\varphi(N)}}\right)\\
&=-R_N^{-1}W_N \left(-R_N^{-1}W_N e_{b_N}+v_N W_N e_{b_{\varphi(N)}}\right)\\
&\quad+v_NW_N \left((R_{\varphi(N)}^{-1}W_{\varphi(N)})^2\right)^{\frac{b_{N+1}-b_N}{2(b_{\varphi(N)+1}-b_{\varphi(N)})}} e_{b_{\varphi(N)}}\\
&=(R_N^{-1}W_N)^2e_{b_N}.
\end{align*}
\end{proof}

In particular, we remark that for operators of $C$-type, we get the mentioned result that $T_{v,w,\varphi,b}^{2(b_{n+1}-b_n)}e_k=e_k$ for every $k\in [b_n,b_{n+1})$ since $R_n=W_n$. We now need to know when the operator $T_{v,w,\varphi,b,R}$ is invertible. In the paper \cite{Menet2}, the invertibility of operators of $C$-type was obtained by requiring that the sequence $(v_n)$ decreases sufficiently rapidly. Two adaptations will be necessary here. First we will not consider operators of $C$-type but operators $T_{v,w,\varphi,b,R}$ with $R_n=1$ for every $n$, and secondly we will need to consider a sequence $(v_n)$ which takes infinitely often the same value in order to get the desired counterexample. Note that it is because of this last condition on $v$ that we have to restrict ourselves to operators on $\ell^1(\mathbb{N})$.

\begin{prop}\label{invertible}
Assume that $R_n=1$ for every $n\ge 0$ and that
\[\lim_{N\to \infty}\sup_{n\in \varphi^{-1}(N)}|v_n|=0 \quad\text{and}\quad \sup_{l\ge 1}\Big(\sum_{m=0}^{m_l-1}\prod_{s=0}^{m}|v_{\varphi^s(l)}|\Big)<\infty\] where $m_l=\min\{s\ge 0:\varphi^s(l)=0\}$.
Then the operator $T_{v,w,\varphi,b,R}$ is invertible on $\ell^1(\mathbb{N})$ and 
\[
T_{v,w,\varphi,b,R}^{-1}\ e_k=
\begin{cases}
 \frac{1}{w_{k}}\, e_{k-1} & \textrm{if}\ k\in (b_{n},b_{n+1}),\; n\geq 0,\\
-\sum_{m=0}^{m_n-1}\Big(\prod_{l=0}^mv_{\varphi^l(n)}\Big) e_{b_{\varphi^{m+1}(n)+1}-1}- e_{b_{n+1}-1} & \textrm{if}\ k=b_{n},\ n\ge 1,\\
 -e_{b_1-1}& \textrm{if}\ k=0.
\end{cases}
\]
\end{prop}
\begin{proof}
We first prove that $T_{v,w,\varphi,b,R}$ is injective. Let $x\in \ell^1(\mathbb{N})$ such that $T_{v,w,\varphi,b,R} x=0$ and $n\ge 0$. It follows that $x_k=0$ for every $k\in [b_{n},b_{n+1}-1)$ and that \[-x_{b_{n+1}-1}+\sum_{m \in \varphi^{-1}(n)\backslash\{0\}}v_m x_{b_{m+1}-1}=0.\]

Assume that there exists $n_0$ such that $|x_{b_{n_0+1}-1}|>\varepsilon>0$. Then we deduce that 
\[\sum_{m \in \varphi^{-1}(n_0)\backslash\{0\}}|v_m| |x_{b_{m+1}-1}|>\varepsilon\]
and thus that 
\[\sum_{m \in \mathcal{N}_1}|x_{b_{m+1}-1}|>\frac{\varepsilon}{S_1}\]
where $\mathcal{N}_1=\varphi^{-1}(n_0)\backslash\{0\}$ and $S_1=\sup_{m\in\mathcal{N}_1}|v_m|$.
By looking at $\sum_{n \in \mathcal{N}_1}x_{b_{n+1}-1}$, we deduce in the same way that 
\[\sum_{n \in \mathcal{N}_1}\sum_{m\in \varphi^{-1}(n)} v_m x_{b_{m+1}-1}=\sum_{n \in \mathcal{N}_1}x_{b_{n+1}-1}\quad\text{and thus that }\quad \sum_{m \in \mathcal{N}_2}|x_{b_{m+1}-1}|>\frac{\varepsilon}{S_1S_2}\]
where $\mathcal{N}_2=\varphi^{-1}(\mathcal{N}_1)$ and $S_2=\sup_{m\in\mathcal{N}_2}|v_m|$. By repeating this argument, we get for every $k\ge 2$,
\[\sum_{m \in \mathcal{N}_k}|x_{b_{m+1}-1}|>\frac{\varepsilon}{\prod_{l=1}^kS_l}\]
where $\mathcal{N}_k=\varphi^{-1}(\mathcal{N}_{k-1})$ and $S_k=\sup_{m\in\mathcal{N}_{k}}|v_m|$. Therefore, since $\inf \mathcal{N}_k\ge k$ for every $k\ge 1$, we have by assumption $\lim_k S_k=0$ and this is then impossible than $x$ belongs to $\ell^1(\mathbb{N})$.\\

The operator $T_{v,w,\varphi,b,R}$ is thus injective and we can compute that
\[
T_{v,w,\varphi,b,R}^{-1}\ e_k=
\begin{cases}
 \frac{1}{w_{k}}\, e_{k-1} & \textrm{if}\ k\in (b_{n},b_{n+1}),\; n\geq 0,\\
-\sum_{m=0}^{m_n-1}\Big(\prod_{l=0}^mv_{\varphi^l(n)}\Big) e_{b_{\varphi^{m+1}(n)+1}-1}- e_{b_{n+1}-1} & \textrm{if}\ k=b_{n},\ n\ge 1,\\
 -e_{b_1-1}& \textrm{if}\ k=0.
\end{cases}
\]
We only show that for every $n\ge 1$, \[T_{v,w,\varphi,b,R}\left(-\sum_{m=0}^{m_n-1}\Big(\prod_{l=0}^mv_{\varphi^l(n)}\Big) e_{b_{\varphi^{m+1}(n)+1}-1}- e_{b_{n+1}-1}\right)=e_{b_n}.\]
Indeed, we have
\begin{align*}
&T_{v,w,\varphi,b,R}\left(-\sum_{m=0}^{m_n-1}\Big(\prod_{l=0}^mv_{\varphi^l(n)}\Big) e_{b_{\varphi^{m+1}(n)+1}-1}- e_{b_{n+1}-1}\right)\\
&\quad=\Big(\prod_{l=0}^{m_n-1}v_{\varphi^l(n)}\Big)e_{0}-\sum_{m= 0}^{m_n-2}\Big(\prod_{l=0}^mv_{\varphi^l(n)}\Big)\left(v_{\varphi^{m+1}(n)}e_{b_{\varphi^{m+2}(n)}}-e_{b_{\varphi^{m+1}(n)}}\right)\\
&\quad\quad-\left(v_{n}e_{b_{\varphi(n)}}-e_{b_n}\right)\\
&\quad=-\sum_{m= 0}^{m_n-2}\Big(\prod_{l=0}^{m+1}v_{\varphi^l(n)}\Big)e_{b_{\varphi^{m+2}(n)}}+\sum_{m= 0}^{m_n-1} \Big(\prod_{l=0}^{m}v_{\varphi^l(n)}\Big)e_{b_{\varphi^{m+1}(n)}}-v_{n}e_{b_{\varphi(n)}}+e_{b_n}\\
&\quad=e_{b_n}.
\end{align*}

We now show that $T_{v,w,\varphi,b,R}$ is surjective. Let $z\in X$. It suffices to show that the sequence $(T_{v,w,\varphi,b,R}^{-1}P_{[0,b_{n+1})}z)_n$ is a Cauchy sequence, where $P_{[0,j)}z=\sum_{k=0}^{j-1}z_ke_k$.

Let $N> n$. We have
\begin{align*}
&\|T_{v,w,\varphi,b,R}^{-1}P_{[0,b_{N+1})}z-T_{v,w,\varphi,b,R}^{-1}P_{[0,b_{n+1})}z\|\\
&\quad=\Big\|\sum_{k=b_{n+1}}^{b_{N+1}-1}z_kT_{v,w,\varphi,b,R}^{-1}e_k\Big\|\\
&\quad\le \sum_{l=n+1}^N\sum_{k=b_l+1}^{b_{l+1}-1}\frac{1}{|w_{k}|}|z_k|+\sum_{l=n+1}^N  \left(
\sum_{m=0}^{m_l-1} \prod_{s=0}^m|v_{\varphi^s(l)}|\right)|z_{b_l}|+ \sum_{l=n+1}^N |z_{b_l}|\\
&\quad\le  \Big(\frac{1}{\inf_k |w_k|}+\sup_{l\ge 1}\Big(\sum_{m=0}^{m_l-1}\prod_{s=0}^{m}|v_{\varphi^s(l)}|\Big)+1\Big)\|P_{[0,b_{N+1})}z-P_{[0,b_{n+1})}z\|.
\end{align*}
We conclude that the sequence $(T_{v,w,\varphi,b,R}^{-1}P_{[0,b_{n+1})}z)_n$ is a Cauchy sequence since $\Big(\frac{1}{\inf_k |w_k|}+\sup_{l\ge 1}\Big(\sum_{m=0}^{m_l-1}\prod_{s=0}^{m}|v_{\varphi^s(l)}|\Big)<\infty$ and thus that $T_{v,w,\varphi,b,R}$ is surjective. Finally, it follows from the open mapping theorem that the inverse of $T_{v,w,\varphi,b,R}$ is continuous. 
\end{proof}

We can remark that an operator $T_{v,w,\varphi,b,R}$ with $R_n=1$ is therefore invertible if the sequence $(\sup_{n\in \varphi^{-1}(N)}|v_n|)_N$ decreases sufficiently rapidly.

\begin{cor}\label{cor inv}
Assume that $R_n=1$ for every $n\ge 0$. If $\sup_{n\in \varphi^{-1}(m)}|v_n|\le\frac{1}{2^m}$ for every $m\ge 0$, then $T_{v,w,\varphi,b,R}$ is invertible.
\end{cor}
\begin{proof}
We have \[\sup_{n\in \varphi^{-1}(N)}|v_n|\le \frac{1}{2^N}\xrightarrow[N\to \infty]{} 0\] and since for every $l\ge 1$, every $0\le m\le m_l$, we have $\varphi^{m_l-m}(l)\ge m$, it follows that
\[\sum_{m=0}^{m_l-1}\prod_{s=0}^{m}|v_{\varphi^s(l)}|\le \sum_{m=0}^{m_l-1}\frac{1}{2^{\varphi^{m+1}(l)}}\le \sum_{m= 0}^{m_l-1}\frac{1}{2^{m_l-m-1}}\le 2.\]
We can then deduce from Proposition~\ref{invertible} that $T_{v,w,\varphi,b,R}$ is invertible.
\end{proof}

\section{A frequently hypercyclic operator whose inverse is not frequently hypercyclic}

Let $n_0=0$ and $n_k=2^{k-1}$ for every $k\ge 1$. We will consider an operator $\T=T_{v,w,\varphi,b,R}$ with the following parameters:
\begin{itemize}
\item for every $n\in [n_k,n_{k+1})$, $\varphi(n)=n-n_k$;
\item for every $m$, every $n\in \varphi^{-1}(m)$, $v_n=2^{-\tau_m}$;
\item for every $k\ge 0$, for every $n\in [n_k,n_{k+1})$, every $i\in (b_n,b_{n+1})$,
\[
w_i=
\begin{cases}
  \frac{1}{2} & \quad\text{if}\ \ b_n< i\le b_n+\eta_n\\
 1 & \quad\text{if}\ \ b_n+\eta_n< i< b_{n+1}-2\delta_n\\
 \frac{1}{2} & \quad\text{if}\ \  b_{n+1}-2\delta_n\le i <b_{n+1}-\delta_n\\
 2& \quad\text{if}\ \ b_{n+1}-\delta_n\le i<b_{n+1}\\
\end{cases}
\]
\item for every $n\ge 0$, $R_n=1$;
\end{itemize}
where $(\tau_m)_{m\ge 0}$ is an increasing sequence of positive integers and for every $k\ge 0$, for every $n\in [n_k,n_{k+1})$,
\[
 \delta_n=\delta^{(k)},\quad \eta_n=\eta^{(k)}\quad \text{and}\quad b_{n+1}-b_n=\Delta^{(k)},\]
where $(\delta^{(k)})_{k\ge 0}$, $(\eta^{(k)})_{k\ge 0}$ and $(\Delta^{(k)})_{k\ge 0}$ are three 
increasing sequences of positive integers satisfying for every $k\ge 0$,
\[2\delta ^{(k)}+\eta^{(k)}<\Delta^{(k)},\quad \text{$\Delta^{(k+1)}$ is a multiple of $2\Delta^{(k)}$} \quad
\text{and}\quad\frac{\eta^{(k)}}{\Delta^{(k)}}=\frac{\eta^{(0)}}{\Delta^{(0)}}.\]

From now, we will denote by $\T$ this operator which depends on the four parameters $(\tau_m)_{m\ge 0}$, $(\delta^{(k)})_{k\ge 0}$, $(\eta^{(k)})_{k\ge 0}$ and $(\Delta^{(k)})_{k\ge 0}$, and we will show that under convenient conditions on these parameters, $\T$ is an invertible frequently hypercyclic operator on $\ell^1(\mathbb{N})$ whose inverse is not frequently hypercyclic. 

Observe already that since the sequence $(\tau_m)_m$ is increasing, it follows from Corollary~\ref{cor inv} that $\T$ is an invertible operator. Moreover, each finite sequence is an eigenvector of $\T^n$ for some $n$.

\begin{prop}\label{propeigen}
For every $n\ge 0$, every $x\in \text{\emph{span}}{\{e_k:k<b_{n+1}\}}$,
\[\T^{2(b_{n+1}-b_n)}x=2^{-2\eta_n}x.\]
\end{prop}
\begin{proof}
By definition of $(w_j)$, we have $W_n=2^{-\eta_n}$ for every $n\ge 0$. Therefore, for every $K\ge 1$, every $n\in [n_K,n_{K+1})$, if $\varphi(n)\in [n_k,n_{k+1})$, we get 
\[(R^{-1}_{\varphi_n}W_{\varphi(n)})^{\frac{b_{n+1}-b_n}{b_{\varphi(n)+1}-b_{\varphi(n)}}}=(2^{-\eta_{\varphi(n)}})^{\frac{b_{n+1}-b_n}{b_{\varphi(n)+1}-b_{\varphi(n)}}}=
(2^{-\eta^{(k)}})^{\frac{\Delta^{(K)}}{\Delta^{(k)}}}=2^{-\eta^{(K)}}=R^{-1}_nW_n\]
since $(\frac{\eta^{(k)}}{\Delta^{(k)}})_k$ is a constant sequence.

It follows from Lemma~\ref{Lem} that for $n\ge 0$, every $j\in [b_n,b_{n+1})$, 
\[\T^{2(b_{n+1}-b_n)}e_j=2^{-2\eta_n}e_j.\]
Finally, if $x=\sum_{m=0}^n\sum_{j=b_m}^{b_{m+1}-1}x_je_j$ and $n\in [n_K,n_{K+1})$, we have, by using the fact that $\Delta^{(k+1)}$ is a multiple of $2\Delta^{(k)}$ for every $k\ge 0$,
\begin{align*}
\T^{2(b_{n+1}-b_n)}x &=\sum_{k=0}^{K-1}\sum_{m=n_k}^{n_{k+1}-1}\sum_{j=b_m}^{b_{m+1}-1}(2^{-2\eta^{(k)}})^{\frac{2\Delta^{(K)}}{2\Delta^{(k)}}}x_je_j+\sum_{m=n_K}^{n}\sum_{j=b_m}^{b_{m+1}-1}2^{-2\eta^{(K)}}x_je_j\\
&=\sum_{k=0}^{K-1}\sum_{m=n_k}^{n_{k+1}-1}\sum_{j=b_m}^{b_{m+1}-1}2^{-2\eta^{(K)}}x_je_j+\sum_{m=n_K}^{n}\sum_{j=b_m}^{b_{m+1}-1}2^{-2\eta^{(K)}}x_je_j\\
&=2^{-2\eta^{(K)}}x=2^{-2\eta_n}x.
\end{align*}
\end{proof}

In order to show that $\T$ is in fact frequently hypercyclic without additional conditions on the parameters $(\tau_m)_{m\ge 0}$, $(\delta^{(k)})_{k\ge 0}$, $(\eta^{(k)})_{k\ge 0}$ and $(\Delta^{(k)})_{k\ge 0}$, we begin by stating the following technical lemma.

\begin{lemma}\label{lem tec}\text{}
\begin{enumerate}
\item For every $y\in c_{00}$, there exists $k_0$ such that for every $k\ge k_0$,
\[\|\T^ky\|\le 2^{-\frac{\eta^{(0)}}{3\Delta^{(0)}}k}.\]
\item For every $K_0\ge 0$, every $N\ge 1$, there exists $K\ge K_0$ such that for every $n\in [n_K,n_K+N)$, every $k\ge 0$, 
\[\|\T^{k} e_{b_n}\|\le 2^{-\frac{\eta^{(0)}}{3\Delta^{(0)}}k}.\]
\end{enumerate}
\end{lemma}
\begin{proof}\text{}
\begin{enumerate}
\item Let $K\ge 0$, $y=\sum_{l=0}^{b_{n_{K+1}}-1}y_le_l$ and $C=\sup_{k<2\Delta^{(K)}}\|\T^k\|$. If we consider a positive integer $L$ and $k\in [2L\Delta^{(K)},2(L+1)\Delta^{(K)})$ then by Proposition~\ref{propeigen},
\begin{align*}
\|\T^ky\|&\le \|\T^{k-2L\Delta^{(K)}}\|\|\T^{2L\Delta^{(K)}}y\|\le C 2^{-2L\eta^{(K)}}\|y\|\\
&\le C 2^{-\frac{k-2\Delta^{(K)}}{\Delta^{(K)}}\eta^{(K)}}\|y\|\le C 2^{-\frac{\eta^{(0)}}{\Delta^{(0)}}k+2\eta^{(K)}}\|y\|
\end{align*}
since $\frac{\eta^{(K)}}{\Delta^{(K)}}=\frac{\eta^{(0)}}{\Delta^{(0)}}$. We can then deduce that there exists $k_0$ such that for every $k\ge k_0$,
\[\|\T^ky\|\le 2^{-\frac{\eta^{(0)}}{3\Delta^{(0)}}k}\]
because 
\[\frac{C 2^{-\frac{\eta^{(0)}}{\Delta^{(0)}}k+2\eta^{(K)}}\|y\|}{2^{-\frac{\eta^{(0)}}{3\Delta^{(0)}}k}}=C2^{-\frac{2\eta^{(0)}}{3\Delta^{(0)}}k+2\eta^{(K)}}\|y\|\xrightarrow[k\to \infty]{} 0.\]
\item Let $K_0\ge 0$, $N\ge 1$ and $C=\sup_{m\le N}\sup_{j\in [0,2(b_{m+1}-b_m))}\|\T^j e_{b_m}\|$. We consider $K\ge K_0$ such that $n_{K+1}-n_K>N$ and  $2^{\frac{\eta^{(K)}}{3}}>C+1$. Let $n\in [n_K,n_{K}+N)$ and $k\ge 0$. If $k< \Delta^{(K)}$ then by definition of $(w_i)$,
\[\|\T^k e_{b_n}\|\le 2^{-\min\{k,\eta^{(K)}\}}\le 2^{-\frac{\eta^{(0)}}{3\Delta^{(0)}}k}\]
since $\frac{\eta^{(0)}}{3\Delta^{(0)}}\le 1$ and $\frac{\eta^{(0)}}{3\Delta^{(0)}}\Delta^{(K)}=\frac{\eta^{(K)}}{3}$.
On the other hand, if $k\in [\Delta^{(K)},2\Delta^{(K)})$, since $\varphi(n)=n-n_K\le N$ and since $\T^{2(b_{m+1}-b_m)}e_{b_m}=e^{-2\eta_m}e_{b_m}$ for every $m\ge 0$, we have, by definition of $\T$,
\begin{align*}
\|\T^k e_{b_n}\|&\le 2^{-\eta^{(K)}}\|\T^{k-\Delta^{(K)}}e_{b_n}\|+|v_n|2^{-\eta^{(K)}}\|\T^{k-\Delta^{(K)}}e_{b_{\varphi(n)}}\|\\
&\le 2^{-\eta^{(K)}}2^{-\frac{\eta^{(0)}}{3\Delta^{(0)}}(k-\Delta^{(K)})}+C 2^{-\eta^{(K)}}\\
&\le (C+1)2^{-\eta^{(K)}} \le 2^{-\frac{2\eta^{(K)}}{3}}\le 2^{-\frac{2\eta^{(0)}}{3\Delta^{(0)}}\Delta^{(K)}}\le 2^{-\frac{\eta^{(0)}}{3\Delta^{(0)}}k}.
\end{align*}
Finally, for every $L\ge 1$, every $k\in [2L\Delta^{(K)},2(L+1)\Delta^{(K)})$, we get
\begin{align*}
\|\T^k e_{b_n}\|&=2^{-2L\eta^{(K)}} \|\T^{k-2L\Delta^{(K)}}e_{b_n}\|\\
&\le 2^{-2L\eta^{(K)}}2^{-\frac{\eta^{(0)}}{3\Delta^{(0)}}(k-2L\Delta^{(K)})}\\
&\le 2^{-2L\Delta^{(K)}\frac{\eta^{(0)}}{\Delta^{(0)}}}2^{-\frac{\eta^{(0)}}{3\Delta^{(0)}}(k-2L\Delta^{(K)})}\le 2^{-\frac{\eta^{(0)}}{3\Delta^{(0)}}k}.
\end{align*}
\end{enumerate}
\end{proof}

We are now able to construct a frequently hypercyclic vector for $\T$.

\begin{prop}\label{fhc}
$\T$ is a frequently hypercyclic opertor on $\ell^1(\mathbb{N})$.
\end{prop}
\begin{proof}
Let $(y^{(j)})_{j\ge 1}$ be a dense sequence in $\ell^{1}(\mathbb{N})$ with $\text{deg}(y^{(j)})<b_{n_{j+1}}$. Let $A(s,l)$ be sets of positive lower density such that for every $j\in A(s,l)$, every $j'\in A(s',l')$, if $j\ne j'$ then $|j-j'|\ge s +s'$ and $\min A(s,l)\ge l$. The construction of such sets can be found in \cite[Lemma 2.5]{Boni}.

We let $x=\sum_{j\ge 1}\sum_{m\in A(s_j,l_j)}x^{(m)}$ where
\[x^{(m)}= \sum_{n=0}^{n_{j+1}-1}\sum_{i=b_n}^{b_{n+1}-1}y^{(j)}_i \frac{2^{-\big(m-i+b_n-2N_j\Delta^{(j)}-2N_j\eta^{(j)}-1\big)}}{v_{n_{k_m}+n} \left(\prod_{t=b_n+1}^{i}w_t\right)}e_{b_{n_{k_m}+n+1}-(m-i+b_n-2N_j\Delta^{(j)})}\]
and where for every $j\ge 1$, every $m\in A(s_j,l_j)$,
\begin{enumerate}
\item $N_j$ is sufficiently big so that
\[\|y^{(j)}\|\frac{2^{-\big(2N_j\Delta^{(j)}-(2N_j+1)\eta^{(j)}\big)}}{\inf\{2^{-\tau_n}:n<n_{j+1}\}}\le \frac{1}{2^j}.\]
\item $s_j$ is sufficiently big so that $s_j> (2N_j+1)\Delta^{(j)}$,
\[\|y^{(j)}\| \frac{2^{-\big(s_j-(2N_j+1)\Delta^{(j)}-(2N_j+1)\eta^{(j)}-1\big)}}{\inf\{2^{-\tau_n}:n<n_{j+1}\}}\le \frac{1}{2^j}\]
and 
\[\|y^{(j)}\|\frac{2^{-\frac{\eta^{(0)}}{3\Delta^{(0)}}s_j+(2N_j+1)\eta^{(j)}+1}}{\inf\{2^{-\tau_n}:n<n_{j+1}\}}\le \frac{1}{2^j},\]
and so that for every $r<n_{j+1}$ and every $n\ge s_j-(2N_j+1)\Delta^{(j)}$,
\[\|\T^n e_{b_{r}}\|\le 2^{-\frac{\eta^{(0)}}{3\Delta^{(0)}}n}.\]
This last condition can be guaranteed thanks to Lemma~\ref{lem tec}.
\item $l_j$ is sufficiently big so that $l_j> (2N_j+1)\Delta^{(j)}$ and
\[\|y^{(j)}\| \frac{2^{-\big(l_j-(2N_j+1)\Delta^{(j)}-(2N_j+1)\eta^{(j)}-2\big)}}{\inf\{2^{-\tau_n}:n<n_{j+1}\}}\le \frac{1}{2^j}.\]
\item $k_m$ is sufficiently big so that $n_{k_m}+n_{j+1}<n_{k_m+1}$, $\eta^{(k_m)}\ge (2N_j+1)\Delta^{(j)}$ and  $\delta^{(k_m)}\ge m$,
and so that for every $N\in [n_{k_m},n_{k_m}+n_{j+1})$, every $n\ge 0$,
\[\|\T^n e_{b_{N}}\|\le 2^{-\frac{\eta^{(0)}}{3\Delta^{(0)}}n},\]
which can be guaranteed thanks to Lemma~\ref{lem tec}.
\end{enumerate}
The rest of the proof consists in showing that $x$ belongs to $\ell^1(\mathbb{N})$ and that for every $m\in A(s_j,l_j)$, ${\|\T^m x-y^{(j)}\|<\varepsilon_j}$ for some sequence $(\varepsilon_j)_j$ tending to $0$. Since $(y^{(j)})_j$ is a dense sequence in $\ell^1(\mathbb{N})$ and each set $A(s_j,l_j)$ has a positive lower density, it will then follow that $x$ is a frequently hypercyclic vector for $\T$.

Let $j\ge 1$ and $m\in A(s_j,l_j)$. We have
\begin{align*}
\|x^{(m)}\|&\le \sum_{n=0}^{n_{j+1}-1}\sum_{i=b_n}^{b_{n+1}-1}|y^{(j)}_i| \frac{2^{-\big(m-i+b_n-2N_j\Delta^{(j)}-2N_j\eta^{(j)}-1\big)}}{|v_{n_{k_m}+n}| 
\left(\prod_{t=b_n+1}^{i}|w_t|\right)}\\
&\le \|y^{(j)}\| \frac{2^{-\big(m-(2N_j+1)\Delta^{(j)}-2N_j\eta^{(j)}-1\big)}}{ \inf\{|v_{n_{k_m}+n}|:n<n_{j+1}\}2^{-\eta^{(j)}}}\\
&\le \|y^{(j)}\| \frac{2^{-\big(m-(2N_j+1)\Delta^{(j)}-(2N_j+1)\eta^{(j)}-1\big)}}{\inf\{2^{-\tau_n}:n<n_{j+1}\}}.
\end{align*}
Moreover, since $\min A(s,l)\ge l$, it follows from (3) that
\begin{align*}
&\sum_{j\ge 1}\sum_{m\in A(s_j,l_j)}\|y^{(j)}\| \frac{2^{-\big(m-(2N_j+1)\Delta^{(j)}-(2N_j+1)\eta^{(j)}-1\big)}}{\inf\{2^{-\tau_n}:n<n_{j+1}\}}\\
&\quad \le \sum_{j\ge 1}\|y^{(j)}\| \frac{2^{-\big(l_j-(2N_j+1)\Delta^{(j)}-(2N_j+1)\eta^{(j)}-2\big)}}{\inf\{2^{-\tau_n}:n<n_{j+1}\}}\le 1.
\end{align*}
We can thus deduce that $x$ is well-defined and belongs to $\ell^1(\mathbb{N})$.\\

Let $J\ge 1$ and $M\in A(s_J,l_J)$. In order to estimate $\|\T^Mx-y^{(J)}\|$, we compute the elements $\T^{M} x^{(m)}$ with $m\in \bigcup_{j\ge 1} A(s_{j},l_{j})$ by dividing our study into three cases: $m=M$, $m>M$ and $m<M$. Let $j\ge 1$ and $m\in A(s_{j},l_{j})$.
\begin{itemize}
\item Case 1 ($m=M$). Let $n< n_{J+1}$ and $i\in [b_n,b_{n+1})$. We have
\begin{align*}
&\T^{M-i+b_n-2N_J\Delta^{(J)}} e_{b_{n_{k_M}+n+1}-(M-i+b_n-2N_J\Delta^{(J)})}\\
&\quad=2^{M-i+b_n-2N_J\Delta^{(J)}-1}(v_{n_{k_M}+n}e_{b_n}-e_{b_{n_{k_M}+n}})
\end{align*}
since $1\le M-i+b_n-2N_J\Delta^{(J)}\le \delta^{(k_M)}$ by (3) and (4) and since $n_{k_M}+n\in [n_{k_M},n_{k_M +1})$ by $(4)$. 
Moreover, since $(2N_J+1)\Delta^{(J)}\le \eta^{(k_M)}$ by (4), we have by Proposition~\ref{propeigen}
\begin{align*}
&\T^{M-i+b_n}e_{b_{n_{k_M}+n+1}-(M-i+b_n-2N_J\Delta^{(J)})}\\
&\quad=2^{M-i+b_n-2N_J\Delta^{(J)}-2N_J\eta^{(J)}-1}v_{n_{k_M}+n}e_{b_n}\\
&\quad\quad-2^{M-i+b_n-4N_J\Delta^{(J)}-1}e_{b_{n_{k_M}+n}+2N_J\Delta^{(J)}}
\end{align*}
 and
\begin{align*}
&\T^{M}e_{b_{n_{k_M}+n+1}-(M-i+b_n-2N_J\Delta^{(J)})}\\
&\quad=2^{M-i+b_n-2N_J\Delta^{(J)}-2N_J\eta^{(J)}-1}\left(\prod_{t=b_n+1}^{i}w_t\right)v_{n_{k_M}+n}e_{i}\\
&\quad\quad-2^{M-2i+2b_n-4N_J\Delta^{(J)}-1}e_{b_{n_{k_M}+n}+2N_J\Delta^{(J)}+i-b_n}.
\end{align*}

It then follows from (1) that
\begin{align*}
&\|\T^Mx^{(M)}-y^{(J)}\|\\
&\quad\le \sum_{n=0}^{n_{J+1}-1}\sum_{i=b_n}^{b_{n+1}-1}|y^{(J)}_i| \frac{2^{-\big(M-i+b_n-2N_J\Delta^{(J)}-2N_J\eta^{(J)}-1\big)}}{|v_{n_{k_M}+n}| \left(\prod_{t=b_n+1}^{i}|w_t|\right)} 2^{M-2i+2b_n-4N_J\Delta^{(J)}-1}\\
&\quad \le \|y^{(J)}\|\frac{2^{-\big(2N_J\Delta^{(J)}-(2N_J+1)\eta^{(J)}\big)}}{\inf\{2^{-\tau_n}:n<n_{J+1}\}}\le\frac{1}{2^J}.
\end{align*}

\item Case 2 ($m>M$). Let $n< n_{j+1}$ and $i\in [b_n,b_{n+1})$. By the properties of $(A(s,l))$ and (2), we have 
\[m-i+b_n-2N_{j}\Delta^{(j)} \ge m-(2N_{j}+1)\Delta^{(j)}> m-s_{j}\ge M\]
and by (4), we have
\[m-i+b_n-2N_{j}\Delta^{(j)}\le m\le\delta^{(k_m)}.\] 
It then follows that
\[\T^M e_{b_{n_{k_m}+n+1}-(m-i+b_n-2N_{j}\Delta^{(j)})}=2^M e_{b_{n_{k_m}+n+1}+M-(m-i+b_n-2N_{j}\Delta^{(j)})}\]
and by using (2), we get

\begin{align*}
&\|\T^M x^{(m)}\|\\
&\quad\le \sum_{n=0}^{n_{j+1}-1}\sum_{i=b_n}^{b_{n+1}-1}|y^{(j)}_i| \frac{2^{-\big(m-i+b_n-2N_j\Delta^{(j)}-2N_j\eta^{(j)}-1\big)}}{|v_{n_{k_m}+n}|\left(\prod_{t=b_n+1}^{i}|w_t|\right)}2^M\\
&\quad \le \frac{\|y^{(j)}\|}{2^{m-M-s_j}} \frac{2^{-\big(s_j-(2N_j+1)\Delta^{(j)}-(2N_j+1)\eta^{(j)}-1\big)}}{\inf\{2^{-\tau_n}:n<n_{j+1}\}}\le \frac{1}{2^{m-M-s_j+j}}.
\end{align*}

\item Case 3 ($m<M$). Let $n< n_{j+1}$ and $i\in [b_n,b_{n+1})$. Since $m-i+b_n-2N_{j}\Delta^{(j)}\le m\le \delta^{(k_m)}$ by (4), we have
\begin{align*}
&\T^{m-i+b_n-2N_{j}\Delta^{(j)}} e_{b_{n_{k_m}+n+1}-(m-i+b_n-2N_{j}\Delta^{(j)})}\\
&\quad=2^{m-i+b_n-2N_{j}\Delta^{(j)}-1}v_{n_{k_m}+n}e_{b_n}-2^{m-i+b_n-2N_{j}\Delta^{(j)}-1}e_{b_{n_{k_m}+n}}
\end{align*}
and since $M-(m-i+b_n-2N_{j}\Delta^{(j)})\ge M-m\ge s_{j}$, we deduce from (2) and (4) that
\begin{align*}
&\|\T^{M} e_{b_{n_{k_m}+n+1}-(m-i+b_n-2N_{j}\Delta^{(j)})}\|\\
&\quad\le\frac{2^{m-i+b_n-2N_{j}\Delta^{(j)}-1}|v_{n_{k_m}+n}|}{2^{\frac{\eta^{(0)}}{3\Delta^{(0)}}(M-m)}}+\frac{2^{m-i+b_n-2N_{j}\Delta^{(j)}-1}}{2^{\frac{\eta^{(0)}}{3\Delta^{(0)}}(M-m)}}\\
&\quad\le \frac{2^{m-i+b_n-2N_{j}\Delta^{(j)}}}{2^{\frac{\eta^{(0)}}{3\Delta^{(0)}}(M-m)}}.
\end{align*}
We can then deduce from (2) that
\begin{align*}
&\|\T^M x^{(m)}\|\\
&\quad\le \sum_{n=0}^{n_{j+1}-1}\sum_{i=b_n}^{b_{n+1}-1}|y^{(j)}_i| \frac{2^{-\big(m-i+b_n-2N_j\Delta^{(j)}-2N_j\eta^{(j)}-1\big)}}{|v_{n_{k_{m}}+n}|\left(\prod_{t=b_n+1}^{i}|w_t|\right)}\frac{2^{m-i+b_n-2N_{j}\Delta^{(j)}}}{2^{\frac{\eta^{(0)}}{3\Delta^{(0)}}(M-m)}}\\
&\quad\le \|y^{(j)}\| \frac{2^{(2N_j+1)\eta^{(j)}+1}}{\inf\{2^{-\tau_n}:n<n_{j+1}\}2^{\frac{\eta^{(0)}}{3\Delta^{(0)}}(M-m)}}\\
&\quad \le \frac{\|y^{(j)}\|}{2^{\frac{\eta^{(0)}}{3\Delta^{(0)}}(M-m-s_j)}}\frac{2^{-\frac{\eta^{(0)}}{3\Delta^{(0)}}s_j+(2N_j+1)\eta^{(j)}+1}}{\inf\{2^{-\tau_n}:n<n_{j+1}\}}\\
&\quad \le \frac{1}{2^{\frac{\eta^{(0)}}{3\Delta^{(0)}}(M-m-s_j)+j}}
\end{align*}
\end{itemize}

In conclusion, thanks to properties of sets $A(s,l)$, we have for every $J\ge 1$, every $M\in A(s_J,l_J)$,
\begin{align*}
&\|T^Mx-y^{(J)}\|\\
&\le \frac{1}{2^J}+ \sum_{j\ge 1}\sum_{\substack{m\in A(s_j,l_j)\\ m>M}} \frac{1}{2^{m-M-s_j+j}}+ \sum_{j\ge 1}\sum_{\substack{m\in A(s_j,l_j)\\m<M}} \frac{1}{2^{\frac{\eta^{(0)}}{3\Delta^{(0)}}(M-m-s_j)+j}}\\
&\le \frac{1}{2^J}+ \sum_{j\ge 1}\sum_{m\ge M+s_j+s_J} \frac{1}{2^{m-M-s_j+j}}+ \sum_{j\ge 1}\sum_{m\le M-s_j-s_J} \frac{1}{2^{\frac{\eta^{(0)}}{3\Delta^{(0)}}(M-m-s_j)+j}}\\
&\le \frac{1}{2^J}+ \sum_{j\ge 1}\sum_{m\ge s_J} \frac{1}{2^{m+j}}+ \sum_{j\ge 1}\sum_{m\ge s_J} \frac{1}{2^{\frac{\eta^{(0)}}{3\Delta^{(0)}}m+j}}\\
&\le \frac{1}{2^J}+ \frac{1}{2^{s_J-1}}+ \frac{1}{2^{\frac{\eta^{(0)}}{3\Delta^{(0)}}s_J}} \left(\sum_{m\ge 0} 2^{-\frac{\eta^{(0)}}{3\Delta^{(0)}}m}\right)\xrightarrow[J\to \infty]{} 0
\end{align*}
We conclude that $x$ is a frequently hypercyclic vector for $\T$.
\end{proof}

It remains to show that $\T^{-1}$ is not frequently hypercyclic under convenient conditions on the parameters $(\tau_m)$, $(\delta^{(k)})$, $(\eta^{(k)})$ and $(\Delta^{(k)})$. The proof of this fact will rely on the study of dynamical behaviours of finite sequences under the action of $\T^{-1}$. Therefore, we begin by a technical lemma concerning finite sequences and given a vector $x\in \ell^1(\mathbb{N})$, we let for every $n\ge 0$ and for every $I\subset \mathbb{N}$ the following notations:
\[P_nx:=\sum_{k=b_n}^{b_{n+1}-1}x_k e_k\quad \text{and} \quad P_Ix=\sum_{n\in I}P_n x.\]

\begin{lemma}\label{lemend} Let $x\in \ell^1(\mathbb{N})$. The following conditions are satisfied:
\begin{enumerate}
\item $\|\T^{-1}x\|\le 2\|x\|$.
\item For every $l\ge 0$, every $s\ge 1$, every $n\in \varphi^{-s}(l)\backslash\{0\}$, every $j\ge 0$, \[\|P_l\T^{-j}P_{n}x\|\le 2^{j-\tau_{l+s-1}} \|P_{n}x\|.\] In particular, for every $l\ge 0$, every $s\ge 1$, every $j\ge 0$,
\[\|P_l\T^{-j}P_{\varphi^{-s}(l)\backslash\{0\}}x\|\le 2^{j-\tau_{l+s-1}} \|P_{\varphi^{-s}(l)\backslash\{0\}}x\|.\]
\item For every $k\ge 0$, every $l\in [n_k,n_{k+1})$, every $j\ge 0$,
\[\|P_{l}\T^{-j}P_lx\|\ge 2^{\eta^{(k)}\left\lfloor \frac{j}{\Delta^{(k)}}\right\rfloor-\delta^{(k)}}\|P_lx\|.\] 
In particular, for every $l\ge 0$, if $P_l x\ne 0$ then 
\[\lim_{j\to \infty}\|P_{l}\T^{-j}P_lx\|=\infty.\]
\end{enumerate}
\end{lemma}
\begin{proof}
Let $x\in \ell^1(\mathbb{N})$.
\begin{enumerate}
\item  By Proposition~\ref{invertible}, we have
\begin{align*}
\|\T^{-1}x\|&\le \sum_{n\ge 0}\sum_{k=b_n+1}^{b_{n+1}-1}|x_k|\|\T^{-1}e_k\| + |x_{b_0}|\|\T^{-1}e_{b_0}\|+ \sum_{n\ge 1}|x_{b_n}|\|\T^{-1}e_{b_n}\|\\
&\le \sum_{n\ge 0}\sum_{k=b_n+1}^{b_{n+1}-1}\frac{|x_k|}{|w_k|} + |x_{b_0}|+\sum_{n\ge 1}|x_{b_n}|
\Big(1+\sum_{m=0}^{m_n-1}(\prod_{l=0}^m |v_{\varphi^l(n)}|)\Big)\\
&\le \sum_{n\ge 0}\sum_{k=b_n+1}^{b_{n+1}-1}2|x_k| + |x_{b_0}|+\sum_{n \ge 1}|x_{b_n}|
\Big(1+\sum_{m=0}^{m_n-1}\frac{1}{2^{\tau_{\varphi(n)}+m}}\Big)
\end{align*}
since $\tau_{j}\ge 1$ for every $j\ge 0$. Finally, since $\sum_{m=0}^{m_n-1}\frac{1}{2^{\tau_{\varphi(n)}+m}}\le \frac{1}{2^{\tau_{\varphi(n)}-1}}\le 1$, we conclude that $\|\T^{-1}x\|\le 2\|x\|$.
\item Let $l\ge 0$, $s\ge 1$, $n\in \varphi^{-s}(l)\backslash\{0\}$, $j \ge 0$ and $k\in [b_n,b_{n+1})$. If $j\le k-b_n$, we have $\|P_l\T^{-j}e_k\|=0$  and if $k-b_n<j\le k+b_{n+1}-2b_n$ then
\begin{align*}
&\|P_l\T^{-j}e_k\|\\
&\quad\le \left(\prod_{t=b_n+1}^{k}|w_t|\right)^{-1}\left\|P_l\T^{-(j-k+b_n-1)}\left(-e_{b_{n+1}-1}
-\sum_{m=0}^{m_n-1}(\prod_{l=0}^m v_{\varphi^l(n)})e_{b_{\varphi^{m+1}(n)+1}-1}\right)\right\|\\
&\quad\le 2^{k-b_n}\left\|\T^{-(j-k+b_n-1)}\left(\sum_{m=0}^{m_n-1}(\prod_{l=0}^m v_{\varphi^l(n)})e_{b_{\varphi^{m+1}(n)+1}-1}\right)\right\|\\
&\quad\le |v_{n}|2^{j-1}\sum_{m=0}^{m_n-1}(\prod_{l=1}^m|v_{\varphi^l(n)}|)\le |v_{n}|2^{j}.
\end{align*}
In addition, if $k+b_{n+1}-2b_n<j<2(b_{n+1}-b_n)$ then $\|P_l\T^{-j}e_k\|=0$. Since $\T^{2(b_{n+1}-b_n)}e_k=2^{-2\eta_{n}}e_k$, we have $\T^{-2(b_{n+1}-b_n)}e_k=2^{2\eta_{n}}e_k$ and since $2^{2\eta_{n}}\le 2^{2(b_{n+1}-b_n)}$, we can deduce that if $j\ge 2(b_{n+1}-b_n)$ then $\|P_l\T^{-j}e_k\|\le |v_{n}|2^{j}$. Therefore, we can write
\[
\|P_l\T^{-j}P_nx\|\le \sum_{k=b_n}^{b_{n+1}-1}|x_k| \|P_l\T^{-j}e_k\| 
\le \sum_{k=b_n}^{b_{n+1}-1}|x_k| |v_{n}|2^{j}\le 2^{j-\tau_{l+s-1}} \|P_{n}x\|
\]
because $(\tau_m)_m$ is increasing and if $n\in \varphi^{-s}(l)\backslash\{0\}$ then $\varphi(n)\ge l+s-1$.
\item Let $k\ge 0$ and $l\in [n_k,n_{k+1})$. We can remark that $P_l\T^{-\Delta^{(k)}}P_lx=-2^{\eta^{(k)}}P_lx$ and it then suffices to prove that for every $j<\Delta^{(k)}$, $\|P_{l}\T^{-j}P_lx\|\ge 2^{-\delta^{(k)}}\|P_lx\|$. Let $j<\Delta^{(k)}$ and $m\in [b_l,b_{l+1})$.
\begin{itemize}
\item If $0\le j\le m-b_l$, we have $\|P_{l}\T^{-j}P_le_m\|=\prod_{t=m-j+1}^{m}|w_t|^{-1}\ge 2^{-\delta^{(k)}}$.
\item If $m-b_l< j< \Delta^{(k)}$, we have \[\|P_{l}\T^{-j}P_le_m\|=\left(\prod_{t=b_l+1}^{m}|w_t|^{-1}\right)\left(\prod_{t=b_{l+1}-j+m-b_l+1}^{b_{l+1}-1}|w_t|^{-1}\right)\] and since $b_{l+1}-j+m-b_l+1> m$, each weight is taken at most one time. We can then deduce from the definition of $(w_i)_i$ that \[\|P_{l}\T^{-j}P_le_m\|\ge 2^{-\delta^{(k)}}.\]
\end{itemize}
\end{enumerate}
\end{proof}

We can now prove that if the sequence $(\tau_m)$ grows sufficiently rapidly then $\T^{-1}$ is not frequently hypercyclic.

\begin{prop}\label{prelim}
Let $S_l=\sum_{l'\le l}(2(b_{l'+1}-b_{l'})+l')$ for every $l\ge 0$ and let $(J_l)_{l\ge 0}$ be a sequence of positive integers such that for every $j\ge J_l$, every $x\in \ell^1(\mathbb{N})$, $\|P_l\T^{-j}P_lx\|\ge  2^{S_l}\|P_lx\|$. If for every $l\ge 0$, 
\[\tau_l\ge S_l+2\eta_l+\delta_l+2l+3 \quad\text{and}\quad \frac{J_l}{\tau_l-l-S_l-\delta_l-3}\le \frac{1}{2^l},\]  
then $\T^{-1}$ is not frequently hypercyclic.
\end{prop}
\begin{proof}
Let  $S_l=\sum_{l'\le l}(2(b_{l'+1}-b_{l'})+l')$ for every $l\ge 0$ and let $(J_l)_{l\ge 0}$ be a sequence of positive integers such that for every $j\ge J_l$, every $x\in \ell^1(\mathbb{N})$, 
\[\|P_l\T^{-j}P_lx\|\ge  2^{S_l}\|P_lx\|.\]
Let $x\in \ell^1(\mathbb{N})$ be a hypercyclic vector for $\T^{-1}$. We can already remark that for every $j\ge 0$, every $n\ge 0$, we have
\begin{equation*}\label{useful}
\|\T^{-j}x\|\ge \|P_{n}\T^{-j}x\|\ge \|P_{n}\T^{-j}P_{n}x\|-\sum_{s\ge 1}\|P_{n}\T^{-j}P_{\varphi^{-s}(n)\backslash\{0\}}x\|.
\end{equation*}
In particular, if $P_nx\ne 0$, it follows that the set $\{j\ge 0:\sum_{s\ge 1}\|P_{n}\T^{-j}P_{\varphi^{-s}(n)\backslash\{0\}}x\|> \frac{\|P_{n}\T^{-j}P_{n}x\|}{4}\}$ is non-empty since 
\[\sum_{s\ge 1}\|P_{n}\T^{-j}P_{\varphi^{-s}(n)\backslash\{0\}}x\|\ge \|P_{n} \T^{-j}P_{n} x\|-\|P_{n} \T^{-j}x\|,\] 
$\|P_{n}\T^{-j}P_{n}x\|$ tends to $\infty$ as $j\to \infty$ (Lemma~\ref{lemend} (3)) and for every $J$, there exists $j\ge J$ such that $\|\T^{-j}x\|\le 1$ since $x$ is hypercyclic. Moreover, since $x$ is hypercyclic for $\T^{-1}$, $x-P_0 x\ne 0$ and we can consider $l_0\ge 1$ such that $\|P_{l_0}x\|\ge \frac{1}{2^{l_0}}\|x-P_0x\|$. Finally, we let $k_0$ such that $l_0\in [n_{k_0},n_{k_0+1})$.

The goal of this proof will be to show that there exists an increasing sequence $(l_m)_{m\ge 0}$ and a sequence $(j_m)_{m\ge 1}$ tending to infinity such that for every $m\ge 1$,
\[\frac{\#\{j< j_{m}:\|\T^{-j}x\|\ge \frac{3}{4}\|P_{l_0}x\|\}}{j_{m}}\ge 1-\frac{1}{2^{l_{m-1}}}.\]
It will then follow that $x$ is not frequently hypercyclic for $\T^{-1}$ since $\|P_{l_0}x\|>0$ and thus that $\T^{-1}$ is not frequently hypercyclic.

We first show that if there exist sequences $(l_m)_{m\ge 0}$ and $(j_m)_{m\ge 1}$ such that for every $m\ge 1$,
\[j_m:=\min\left\{j\ge 0:\sum_{s\ge 1}\|P_{l_{m-1}}\T^{-j}P_{\varphi^{-s}(l_{m-1})}x\|> \frac{\|P_{l_{m-1}}\T^{-j}P_{l_{m-1}}x\|}{4}\right\}\]
and such that for every $m\ge 1$, there exists $s_m\ge 1$ so that
\begin{itemize}
\item $l_m\in \varphi^{-s_m}(l_{m-1})$,
\item $\|P_{l_{m-1}}\T^{-j_m}P_{\varphi^{-s_m}(l_{m-1})}x\|>\frac{\|P_{l_{m-1}}\T^{-j_m}P_{l_{m-1}}x\|}{2^{s_m+2}}$,
\item $2^{S_{l_m}}\|P_{l_{m}}x\|\ge \|P_{l_0}x\|$,
\end{itemize}
then $(l_m)_{m\ge 0}$ is an increasing sequence,  $(j_m)_{m\ge 1}$ tends to infinity and for every $m\ge 1$,
\[\frac{\#\{j< j_{m}:\|\T^{-j}x\|\ge \frac{3}{4}\|P_{l_0}x\|\}}{j_{m}}\ge 1-\frac{1}{2^{l_{m-1}}}.\]

Let $m\ge 1$. Note that $j_m$ is well-defined since $P_{l_{m-1}}x\ne 0$. Indeed, since $\|P_{l_0}x\|>0$ and since for every $n\ge 1$, 
\[2^{S_{l_n}}\|P_{l_{n}}x\|\ge \|P_{l_0}x\|\]
it follows that $P_{l_{m-1}}x\ne 0$. We can also remark that the sequence $(l_m)_{m\ge 1}$ is increasing since $l_m\in \varphi^{-s_m}(l_{m-1})$ and $s_m\ge 1$.

On the other hand, we have \[j_m>\tau_{l_{m-1}+s_m-1}-l_0-S_{l_{m-1}}-\delta_{l_{m-1}}-s_{m}-2\] since for every $j\le \tau_{l_{m-1}+s_m-1}-l_0-S_{l_{m-1}}-\delta_{l_{m-1}}-s_{m}-2$, we have by Lemma~\ref{lemend}~(2)-(3)
 \begin{align*} 
 \|P_{l_{m-1}}\T^{-j}P_{\varphi^{-s_m}(l_{m-1})}x\|&\le 2^{j-\tau_{l_{m-1}+s_m-1}}\|P_{\varphi^{-s_m}(l_{m-1})}x\|\\
 &\le 2^{j-\tau_{l_{m-1}+s_m-1}} \|x-P_0 x\|\\
 &\le 2^{j-\tau_{l_{m-1}+s_m-1}+l_0} \|P_{l_0} x\|\\
 &\le 2^{j-\tau_{l_{m-1}+s_m-1}+l_0+S_{l_{m-1}}}\|P_{l_{m-1}}x\|\\
 &\le 2^{j-\tau_{l_{m-1}+s_m-1}+l_0+S_{l_{m-1}}+\delta_{l_{m-1}}}\|P_{l_{m-1}}\T^{-j}P_{l_{m-1}}x\|\\
 &\le \frac{\|P_{l_{m-1}}\T^{-j}P_{l_{m-1}}x\|}{2^{s_m+2}}.
 \end{align*}
 In particular, we have
 \begin{align*}
 j_m&\ge S_{l_{m-1}+s_m-1}+\delta_{l_{m-1}+s_m-1}+2l_{m-1}+2s_m+1-l_0-S_{l_{m-1}}-\delta_{l_{m-1}}-s_{m}-2\\
 &\ge l_{m-1}
 \end{align*} since $(S_j)_j$, $(\delta_j)_j$ and $(l_j)_j$ are increasing, and thus the sequence $(j_m)_{m\ge 1}$ tends to infinity. By assumption, we also have $\|P_{l_{m-1}}\T^{-j}P_{l_{m-1}}x\|\ge  2^{S_{l_{m-1}}}\|P_{l_{m-1}}x\|$ for every $j\ge J_{l_{m-1}}$
 and \begin{align*}
 \frac{j_m-J_{l_{m-1}}}{j_m}&\ge 1-\frac{J_{l_{m-1}}}{\tau_{l_{m-1}+s_m-1}-l_0-S_{l_{m-1}}-\delta_{l_{m-1}}-s_m-2}\\
 &\ge 1-\frac{J_{l_{m-1}}}{\tau_{l_{m-1}+s_m-1}-l_{m-1}-S_{l_{m-1}}-\delta_{l_{m-1}}-s_m-2}\\
 &\ge 1-\frac{J_{l_{m-1}}}{\tau_{l_{m-1}}-l_{m-1}-S_{l_{m-1}}-\delta_{l_{m-1}}-3}\\
 &\ge 1-\frac{1}{2^{l_{m-1}}}
 \end{align*}
since $\tau_{l_{m-1}+s_m-1}\ge \tau_{l_{m-1}}+s_m-1$.
We can then conclude that
\[\frac{\#\{j< j_m:\|\T^{-j}x\|\ge \frac{3}{4}\|P_{l_0}x\|\}}{j_m}\ge 1-\frac{1}{2^{l_{m-1}}}\]
since for every $J_{l_{m-1}}\le j< j_m$, by definition of $j_m$,
\begin{align*}
\|\T^{-j}x\|&\ge \|P_{l_{m-1}}\T^{-j}P_{l_{m-1}}x\|-\sum_{s=1}^{\infty}\|P_{l_{m-1}}\T^{-j}P_{\varphi^{-s}(l_{m-1})}x\|\\
&\ge \frac{3}{4}\|P_{l_{m-1}}\T^{-j}P_{l_{m-1}}x\|\\
&\ge \frac{3}{4}2^{S_{l_{m-1}}}\|P_{l_{m-1}}x\|\\
&\ge \frac{3}{4}\|P_{l_{0}}x\|.
\end{align*}

In order to conclude the proof, it remains to show that there exist sequences $(l_m)_{m\ge 0}$ and $(j_m)_{m\ge 1}$ such that for every $m\ge 1$,
\[j_m:=\min\left\{j\ge 0:\sum_{s\ge 1}\|P_{l_{m-1}}\T^{-j}P_{\varphi^{-s}(l_{m-1})}x\|> \frac{\|P_{l_{m-1}}\T^{-j}P_{l_{m-1}}x\|}{4}\right\}\]
and such that for every $m\ge 1$, there exists $s_m\ge 1$ so that
\begin{itemize}
\item $l_m\in \varphi^{-s_m}(l_{m-1})$,
\item $\|P_{l_{m-1}}\T^{-j_m}P_{\varphi^{-s_m}(l_{m-1})}x\|>\frac{\|P_{l_{m-1}}\T^{-j_m}P_{l_{m-1}}x\|}{2^{s_m+2}}$,
\item $2^{S_{l_{m}}}\|P_{l_{m}}x\|\ge \|P_{l_0}x\|$.
\end{itemize}

Assume that $(l_t)_{0\le t\le m-1}$ and $(j_t)_{1\le t\le m-1}$ have been chosen and satisfy the above conditions. Let $j_m=\min\{j\ge 0:\sum_{s\ge 1}\|P_{l_{m-1}}\T^{-j}P_{\varphi^{-s}(l_{m-1})}x\|> \frac{\|P_{l_{m-1}}\T^{-j}P_{l_{m-1}}x\|}{4}\}$. By definition of $j_m$, there exists $s_m\ge 1$ such that 
\[\|P_{l_{m-1}}\T^{-j_m}P_{\varphi^{-s_m}(l_{m-1})}x\|>\frac{\|P_{l_{m-1}}\T^{-j_m}P_{l_{m-1}}x\|}{2^{s_m+2}}\]
and there exists $l_m\in \varphi^{-s_m}(l_{m-1})$ such that 
\[\|P_{l_{m-1}}\T^{-j_m}P_{l_m}x\|>\frac{\|P_{l_{m-1}}\T^{-j_m}P_{l_{m-1}}x\|}{2^{l_m+s_m+3}}.\]
It remains to show that $2^{S_{l_{m}}}\|P_{l_{m}}x\|\ge \|P_{l_0}x\|$. Let $k_m$ such that $l_m\in [n_{k_m},n_{k_m+1})$. 
 It follows from Lemma~\ref{lemend} (3) that
 \[\|P_{l_{m-1}}\T^{-j_m}P_{l_{m-1}}x\|\ge 2^{\eta^{(k_{m-1})}\left\lfloor\frac{j_m}{\Delta^{(k_{m-1})}}\right\rfloor-\delta^{(k_{m-1})}} \|P_{l_{m-1}}x\|. \] On the other hand, since $\T^{-\left\lfloor \frac{j_m}{2\Delta^{(k_m)}}\right\rfloor 2\Delta^{(k_m)}}P_{l_m}x=2^{2\eta^{(k_m)} \left\lfloor \frac{j_m}{2\Delta^{(k_m)}}\right\rfloor}P_{l_m}x$, it follows from Lemma~\ref{lemend}~(2) that
\begin{align*}
\|P_{l_{m-1}}\T^{-j_m}P_{l_m}x\|&\le 2^{2\eta^{(k_m)} \left\lfloor \frac{j_m}{2\Delta^{(k_m)}}\right\rfloor}\sup_{0\le j< 2\Delta^{(k_m)}}\|P_{l_{m-1}}\T^{-j}P_{l_m}x\|\\
&\le 2^{2\eta^{(k_m)}\frac{j_m}{2\Delta^{(k_m)}}}2^{2\Delta^{(k_m)}-\tau_{l_{m-1}+s_m-1}}\|P_{l_m}x\|.
\end{align*}
Since $\tau_{l_{m-1}+s_m-1}\ge 2\eta_{l_{m-1}+s_m-1}+\delta_{l_{m-1}+s_m-1}+s_m+3$, we deduce from three previous inequalities that 
\begin{align*}
\|P_{l_m}x\|&\ge \frac{2^{2\eta^{(k_{m-1})}\left(\frac{j_m}{\Delta^{(k_{m-1})}}-1\right)-\delta^{(k_{m-1})}}}{2^{2\eta^{(k_m)}\frac{j_m}{2\Delta^{(k_m)}}}2^{2\Delta^{(k_m)}-\tau_{l_{m-1}+s_m-1}}2^{l_m+s_m+3}}\|P_{l_{m-1}}x\|\\
&=\frac{2^{\tau_{l_{m-1}+s_m-1}}}{2^{2\eta^{(k_{m-1})}+\delta^{(k_{m-1})}+2\Delta^{(k_m)}+l_m+s_m+3}}\|P_{l_{m-1}}x\|\quad\text{since $\frac{\eta^{(k_{m-1})}}{\Delta^{(k_{m-1})}}=\frac{\eta^{(k_m)}}{\Delta^{(k_m)}}$}\\
&\ge \frac{1}{2^{2\Delta^{(k_m)}+l_m}}\|P_{l_{m-1}}x\|\\
&\ge \frac{1}{2^{2\Delta^{(k_m)}+l_m}} 2^{-S_{l_{m-1}}}\|P_{l_{0}}x\|\\
&\ge 2^{-S_{l_{m}}}\|P_{l_{0}}x\|\qquad\text{(by definition of $S_l$)}.
\end{align*}
\end{proof}

We conclude this paper by showing that it is possible to find parameters satisfying all required properties in order to obtain the first example of invertible frequently hypercyclic operator whose inverse is not frequently hypercyclic.

\begin{theorem}
There exists an invertible frequently hypercyclic operator $T$ on $\ell^{1}(\mathbb{N})$ such that $T^{-1}$ is not frequently hypercyclic.
\end{theorem}
\begin{proof}
Let $n_0=0$ and $n_k=2^{k-1}$ for every $k\ge 1$. For every $k\ge 0$, we let 
\[\Delta^{(k)}=8^{k+1}\quad \text{and}\quad \eta^{(k)}=\delta^{(k)}=8^{k}\]
so that
\[2\delta ^{(k)}+\eta^{(k)}<\Delta^{(k)},\quad \text{$\Delta^{(k+1)}$ is a multiple of $2\Delta^{(k)}$} \quad
\text{and}\quad\frac{\eta^{(k)}}{\Delta^{(k)}}=\frac{\eta^{(0)}}{\Delta^{(0)}}.\]
Let $(\tau_n)_{n\ge 0}$ be an increasing sequence of positive integers and $T$ the operator $T_{v,w,\varphi,b,R}$ such that
\begin{itemize}
\item for every $n\in [n_k,n_{k+1})$, $\varphi(n)=n-n_k$;
\item for every $m$, every $n\in \varphi^{-1}(m)$, $v_n=2^{-\tau_m}$;
\item for every $k\ge 0$, for every $n\in [n_k,n_{k+1})$, every $i\in (b_n,b_{n+1})$,
\[
w_i=
\begin{cases}
  \frac{1}{2} & \quad\text{if}\ \ b_n< i\le b_n+\eta_n\\
 1 & \quad\text{if}\ \ b_n+\eta_n< i< b_{n+1}-2\delta_n\\
 \frac{1}{2} & \quad\text{if}\ \  b_{n+1}-2\delta_n\le i <b_{n+1}-\delta_n\\
 2& \quad\text{if}\ \ b_{n+1}-\delta_n\le i<b_{n+1}\\
\end{cases}
\]
\item for every $n\ge 0$, $R_n=1$;
\end{itemize}
and such for every $k\ge 0$, for every $n\in [n_k,n_{k+1})$,
\[
 \delta_n=\delta^{(k)},\quad \eta_n=\eta^{(k)}\quad \text{and}\quad b_{n+1}-b_n=\Delta^{(k)}.\]
The operator $T$ is well-defined, invertible (Corollary \ref{cor inv}) and frequently hypercyclic (Proposition~\ref{fhc}) for any choice of $(\tau_n)_{n\ge 0}$. Moreover, for every $x\in \ell^1(\mathbb{N})$, every $l\ge 0$, the sequence $(P_l\T^{-j}P_lx)_{j\ge 0}$ does not depend on $(\tau_n)_{n\ge 0}$. Let $S_l=\sum_{l'\le l}(2(b_{l'+1}-b_{l'})+l')$ for every $l\ge 0$. By using Lemma~\ref{lemend}~(3), we can thus find a sequence $(J_l)_{l\ge 0}$ such that for every increasing sequence $(\tau_n)_{n\ge 0}$, every $j\ge J_l$ and every $x\in \ell^1(\mathbb{N})$, 
 \[\|P_l T^{-j}P_lx\|\ge  2^{S_l}\|P_lx\|.\]
By choosing for $(\tau_n)_{n\ge 0}$ a rapidly increasing sequence so that for every $l\ge 0$, 
\[\tau_l\ge S_l+2\eta_l+\delta_l+2l+3 \quad\text{and}\quad \frac{J_l}{\tau_l-l-S_l-\delta_l-3}\le \frac{1}{2^l},\]  
we can then deduce from Proposition~\ref{prelim} that $T^{-1}$ is not frequently hypercyclic.
\end{proof}


\begin{thebibliography}{HD}

\normalsize
\baselineskip=17pt

\bibitem{Ansari} S. I. Ansari, \emph{Existence of hypercyclic operators on topological vector spaces}, J. Funct. Anal. 148 (1997), 384--390.

\bibitem{BG1} F. Bayart and S. Grivaux, \emph{Hypercyclicité: le rôle du spectre ponctuel unimodulaire}, C. R. Math. Acad. Sci. Paris 338 (2004), 703--708.

\bibitem{BG2} F. Bayart and S. Grivaux, \emph{Frequently hypercyclic operators}. Trans. Amer. Math. Soc. 358 (2006), 5083--5117.

\bibitem{BG3} F. Bayart and S. Grivaux, \emph{Invariant Gaussian measures for operators on Banach spaces and linear dynamics}, Proc. Lond. Math. Soc. (3) 94 (2007), 181--210.

\bibitem{Ruzsa} F. Bayart and I. Ruzsa, \emph{Difference sets and frequently hypercyclic weighted shifts}, Ergod. Th. \& Dynam. Sys. 35 (2015), 691--709.

\bibitem{Book1}
F. Bayart and \'E. Matheron, \emph{Dynamics of linear operators}, Cambridge Tracts in Mathematics 179, Cambridge University Press, 2009.

\bibitem{Bernal} L. Bernal-González, \emph{On hypercyclic operators on Banach spaces}, Proc. Amer.Math. Soc. 127 (1999), 1003--1010.

\bibitem{Boni} A. Bonilla and K.-G. Grosse-Erdmann, \emph{Frequently hypercyclic operators and vectors}, Ergodic Theory Dynam. Systems 27 (2007), 383--404. Erratum: Ergodic Theory Dynam. Systems 29 (2009), 1993--1994.

\bibitem{Monster} S. Grivaux, E. Matheron and Q. Menet, \emph{Linear dynamical systems on Hilbert spaces: typical properties and explicit examples}, Mem. Amer. Math. Soc, to appear.

\bibitem{Book2} K.-G. Grosse-Erdmann and A. Peris, \emph{Linear Chaos}, Universitext, Springer, 2011.

\bibitem{Grosse} K.-G. Grosse-Erdmann, \textit{Frequently hypercyclic bilateral shifts}, Glasgow Mathematical Journal 61 (2019), 271--286.

\bibitem{Guirao} A. J. Guirao, V. Montesinos and V. Zizler, \emph{Open problems in the geometry and analysis of Banach spaces}, Springer, 2016.

\bibitem{Menet} Q. Menet, \emph{Linear chaos and frequent hypercyclicity}, Trans. Amer. Math. Soc. 369 (2017), 4977-4994.

\bibitem{Menet2} Q. Menet, \emph{Inverse of $\mathcal{U}$-frequently hypercyclic operators}, arXiv:1904.07485.

\bibitem{Shkarin} S. Shkarin, \emph{On the spectrum of frequently hypercyclic operators}, Proc. Amer. Math. Soc. 137 (2009), 123--134.

\end{thebibliography}
\end{document}